\documentclass[reqno,12pt]{amsart}
\usepackage{amsmath,amsthm,amssymb}

\date{}

\newtheorem{theorem}{Theorem}[section]
\newtheorem{lemma}[theorem]{Lemma}
\newtheorem{corollary}[theorem]{Corollary}
\newtheorem{proposition}[theorem]{Proposition}

\newtheorem{remark}[theorem]{Remark}
\newtheorem{definition}[theorem]{Definition}

\numberwithin{equation}{section}

\title[Transportation functionals and log-Minkowski problem]{Mass transportation functionals on the sphere with applications to the logarithmic Minkowski problem }
\author[Alexander Kolesnikov]{Alexander V. Kolesnikov}
\address{National Research University Higher School of Economics, Russian Federation}
\email{Sascha77@mail.ru}
\subjclass[2010]{Primary: 52A40, 90C08} 
\keywords{convex bodies, optimal transportation, Kantorovich duality,  log-Minkowski problem, K{\"a}hler-Einstein equation }

\thanks{
The  author was supported by  RFBR
 project 17-01-00662 and  DFG project RO 1195/12-1. This work has been funded by the  Russian Academic Excellence Project '5-100' and supported in part by the Simons Foundation.
}

\begin{document}
  
  \begin{abstract}
  We study the transportation problem on the unit sphere  $S^{n-1}$ for symmetric probability measures and the cost function $c(x,y) = \log \frac{1}{\langle x, y \rangle}$.
  We calculate the variation of the corresponding Kantorovich functional $K$ and study a naturally associated metric-measure space on $S^{n-1}$ endowed with a Riemannian
  metric generated by the corresponding transportational potential.  We introduce a new transportational functional which minimizers  are 
  solutions to the symmetric log-Minkowski problem and prove that $K$ satisfies the following analog of the Gaussian transportation inequality for the uniform probability measure ${\sigma}$ on $S^{n-1}$:
  $\frac{1}{n} Ent(\nu) \ge K({\sigma}, \nu)$. It is shown that there exists a remarkable similarity between our results and the theory of the K{\"a}hler-Einstein equation on Euclidean space.
  As a by-product we obtain a new proof of uniqueness of solution to the log-Minkowski problem for the uniform measure.  
  \end{abstract}
  \maketitle

  \section{Introduction}
  
  We start with explanations and representation of some related results in the Euclidean case.
  Let $\mu=e^{-V} dx$, $\nu = e^{-W} dx$ be probability measures on $\mathbb{R}^n$ and $x \to \nabla \Phi (x)$ be the optimal transportation mapping
  pushing  forward $\mu$ onto $\nu$. The potential $\Phi$ is  a convex function solving the dual Kantorovich problem with quadratic cost. If the densities of 
  measures are sufficiently regular, $\Phi$ solves the related  Monge--Amp{\`e}re equation
  \begin{equation}
  \label{mae}
  e^{-V} = e^{-W(\nabla \Phi)} \det D^2 \Phi
  \end{equation} (see \cite{Villani}, \cite{BoKo} for details).
  The related Kantorovich functional
  $$
  W^2_2(\mu,\nu)  = \inf_{\pi \in \Pi(\mu,\nu)} \int_{\mathbb{R}^n \times \mathbb{R}^n} |x-y|^2 d \pi
  $$
  induces a metric on the space of probability measures with finite second moments and satisfies $W_2^2(\mu,\nu)  = \int |x - \nabla \Phi(x)|^2 d \mu$;
  here $\Pi(\mu,\nu)$  is the space of measures on $\mathbb{R}^n \times \mathbb{R}^n$ with marginals $\mu, \nu$ .
  
  Numerous powerful applications of the optimal transportation theory
  are based on the use of differential structures  on the space $\mathcal{P}_2$ of probability measures  endowed with  metric $W_2$.
  It is by now a classical fact that many  evolutionary equations can be interpreted as gradients flows on $\mathcal{P}_2$. The reader can find a comprehensive representation in  \cite{AGS}
  (see also  \cite{Villani}, \cite{Villani2}, \cite{BoKo}).
   An important related notion which was introduced by R.~McCann is the displacement convexity property, which means 
  convexity along the geodesics in the Kantorovich metric.
  
  To show displacement  convexity of a functional one has to compute its second order derivatives along the geodesics. 
  The corresponding calculus relies on the use of the Monge-Amp\`ere operator and its linearized versions.
  One of such versions is given
 by the following formula:
 \begin{equation}
 \label{Lf}
 L f = {\rm Tr} (D^2 \Phi)^{-1} D^2 f  - \langle \nabla f, \nabla W(\nabla \Phi) \rangle.
 \end{equation}
 This operator  naturally appears with differentiation of  (\ref{mae}).
 In particular, the simple linear variation of the source measure $\mu_{\varepsilon} = \mu (1 + \varepsilon v)$  by a function $v$ with zero mean
 corresponds to the variation of the potential $\Phi_{\varepsilon}  = \Phi + \varepsilon u + o(\varepsilon)$, where 
 $$
 Lu = v.
 $$
  Connection of this formula  to  differential calculus on $\mathcal{P}_2$  is explained in Section~2.
  It was observed in \cite{Kol} that
  $L$ is the generator 
 of the symmetric  Dirichlet form
  $$
  \mathcal{E}(f,g) = \int \langle (D^2 \Phi)^{-1} \nabla f, \nabla g \rangle d \mu = - \int L f g d \mu.
  $$ 
  Let us endow  $\mathbb{R}^n$ with the Riemannian metric $D^2 \Phi$.
    The related metric-measure space $(\mu, D^2 \Phi)$ is a natural geometric and probabilistic object, it  has been studied in 
   \cite{CK}, \cite{Fathi}, \cite{K_moment}, \cite{K_part_I}, \cite{KK}, \cite{KK2}, \cite{KK3}, \cite{Kol}, \cite{KolMil}.

  Of particular interest is the following special case:
  \begin{equation}
  \label{moment-m}
  \mu = e^{-\Phi} dx.
  \end{equation}
  Following the teminology from \cite{CK} we say that $\nu$ is a moment measure if there exists another probability measure 
  $\mu$ of the form  (\ref{moment-m}) such that $\nu$ is the image of $\mu$ under $\nabla \Phi$.
  The most general sufficient condition for $\nu$ to be a moment measure was established by B.~Klartag and D.~Cordero-Erausquin  in \cite{CK}.
  
  It   is known that 
$\Phi$ is the unique maximum point of the following functional:
\begin{equation}
\label{prekopa0}
J(f) = \log \int e^{-f^*} dx  - \int f d \nu,
\end{equation}
where $f^*$ is the Legendre transform of $f$. This fact was used in \cite{CK} to establish well-posedness
of the moment measure problem. 
Another natural functional which minimizers solve the same problem was suggested by F.~Santambrogio in \cite{S}. The following  Gaussian version of this functional was  studied in \cite{KolKos}:
\begin{equation}
\label{GaussMM}
\mathcal{F} (g \cdot \gamma) = {Ent}_{\gamma}(g) - \frac{1}{2} W^2_2(g \cdot \gamma ,\nu),
\end{equation}
  where $\gamma$ is the standard Gaussian measure, $ {Ent}_{\gamma}(g) = \int g \log g d \gamma$ is the Gaussian entropy,
  and $W_2(g \cdot \gamma ,\nu)$  is the Kantorovich distance between $g \cdot \gamma$ and $\nu$.
  In the  particular case $\nu = \gamma$ 
  the minimum of $\mathcal{F}$ equals zero, it is attained at $g=1$. The positivity of the functional 
  $
\mathcal{F} (g \cdot \gamma) 
  $
 for $\nu = \gamma$  is equivalent to the Talagrand transportation inequality
  \begin{equation}
  \label{Tal}
{Ent}_{\gamma}(g) \ge \frac{1}{2} W^2_2(g \cdot \gamma,  \gamma)
  \end{equation}
  (see \cite{BGL}).  See \cite{KolKos}, \cite{Fathi2} for further information on relations between (\ref{GaussMM}) and stability estimates for Gaussian inequalities
  and \cite{FGJ} for variational approach to other transportation inequalities.
  The main result of \cite{Fathi2} is the following estimate:
 $$\mathcal{F} (g \cdot \gamma) \ge - {Ent}_{\gamma}(\nu)$$ provided either $\nu$ or $g \cdot \gamma$
 has zero mean.
 
 In this work we develop a similar formalism on the unit sphere $S^{n-1} \subset \mathbb{R}^n$. Instead of Riemannian analog of $W_2$ on  $S^{n-1}$ 
 we shall work with the functional
    $$
 K(\mu,\nu) = \min_{\pi \in \Pi(\mu,\nu)} \int_{(S^{n-1})^2} c(x,y) d \pi,
 $$ 
 where
 $$
 c(x,y) = \left\{\begin{array}{cc}
             \log \frac{1}{\langle x, y \rangle}, & \langle x, y \rangle > 0\\
             + \infty, &  \langle x, y \rangle \le 0.
           \end{array}\right.
 $$ 
 and $\mu,\nu$ are probability measures on $S^{n-1}.$
  
 The importance of this functional for convex geometry was revealed by V.~Oliker in \cite{Oliker}.
He proved that  the Kantorovich problem with the cost function $c$
is in a sense equivalent to a classical problem from convex geometry, the so-called Aleksandrov problem.
See also \cite{Bert}, \cite{OlikerGangbo}, \cite{KKN}.
The corresponding optimal transportation mapping has the form
  \begin{equation}
  \label{h-tran}
  T(x) = \frac{h(x) \cdot x + \nabla_{S^{n-1}} h(x)}{\sqrt{h^2(x) + |\nabla_{S^{n-1}} h(x)|^2}},
  \end{equation}
  where $h$ is a support function of a convex body contaning the origin. 
 
 In Section 3 we prove that  (\ref{Lf})  admits the following spherical analog for a couple of probability measures with densities
  $$
   \mu = e^{-V} \cdot \sigma, \ \ \nu = e^{-W} \cdot \sigma,
   $$
   and the corresponding optimal transportation mapping (\ref{h-tran}):
$$ L_{\mu,\nu} \Bigl( \frac{u}{h}\Bigr)  = {\rm Tr}  (D^2 h)^{-1} D^2 u - \Big\langle \nabla_{S^{n-1}} W(T) + n T,
   \frac{ux + \nabla_{S^{n-1}} u}{\sqrt{h^2 + |\nabla_{S^{n-1}} h|^2}} \Big\rangle + \frac{u}{h}.
   $$
   Here $\sigma$  is the probability uniform measure on $S^{n-1}$,
   $$D^2 f = f  \cdot {\rm Id} + \nabla^2_{S^{n-1}} f,$$ $\nabla_{S^{n-1}} $ is the gradient on $S^{n-1}$, and $\nabla^2_{S^{n-1}}$ is the Hessian operator on 
   $S^{n-1}$.
   The associated Dirichlet form : 
    $$
\mathcal{E}_{\mu,\nu}(f,g) = \int_{S^{n-1}}  h { \Big\langle ( D^2 h)^{-1} \nabla_{S^{n-1}}  f, \nabla_{S^{n-1}} g  \Big\rangle} d \mu  = -  \int_{S^{n-1}} f  L_{\mu,\nu} g d \mu.
 $$
    A particular case of  $L_{\mu,\nu}$ has been studied by E.~Milman and the author in \cite{KM} in respect to the log-Brunn-Minkowski conjecture
(see works   of B\"or\"oczky, Lutwak, Yang, and Zhang  \cite{BLYZ0}, \cite{BLYZ}).
It is conjectured that a reinforcement of the Brunn--Minkowski inequality
holds within the class of all symmetric convex bodies. We don't  give a complete list of references 
concerning this problem, because it is too long. The readers are advised to consult the papers 
\cite{BLYZ0},  \cite{BLYZ}, \cite{CLM}, \cite{CFM},  \cite{KM}, \cite{Saroglou} and the references therein.
The  log-Brunn--Minkowski inequality implies that the log-Minkowski problem has a unique solution.

{\bf Log-Minkowski problem:}
Given a (symmetric) probability measure $\mu$ on $S^{n-1}$  find a (symmetric) convex body $\Omega \ni 0$ of volume $1$ such that  $\mu$ is the cone measure of $\Omega$.

Note that under additional assumption that $\mu$ has a (sufficiently regular) density
 $\mu = \rho_{\mu} dx$  the support function $h$ of $\Omega$ must satisfy the following equation of the   Monge--Amp{\`e}re type:
\begin{equation}
\label{mucone}
\rho_{\mu} = \frac{1}{n} {h \det D^2 h}.
\end{equation}

The main result of \cite{KM} states that a local version of the even log-Brunn--Minkowski conjecture
is equivalent to the second eigenvalue problem for the operator  $ L_{\mu,\nu} $, where $\mu$  is given by
(\ref{mucone})  and $\nu = \mu \circ T^{-1}$, where $T$ is given by (\ref{h-tran}).
It this case $L_{\mu,\nu}$ has the form
 $$
   L_{\mu,\nu} \Bigl( \frac{u}{h} \Bigr) =  {\rm Tr}  (D^2 h)^{-1} D^2 u - (n-1) \frac{u}{h}.
$$
From now let us restrict ourselves to the symmetric case: it is assumed that all the sets are symmetric and functions are even.
A necessary and sufficient condition (subspace concentration condition) for existence of a solution to the even  log-Minkowski probem  has been established in 
\cite{BLYZ}. It turns out that any minimum point of the functional
\begin{equation}
\label{logh}
h \to \int_{S^{n-1}} \log h d \mu
\end{equation}
with the constraint $|\Omega_h|=1$ is a solution to the log-Minkowski problem for $\mu$.

Following the idea from \cite{S} we introduce another functional which minimizers are solutions to the even log-Minkowski problem.
The spherical entropy of a measure $m$ is defined as follows:
$$
 Ent(m) = \left\{\begin{array}{cc}
             \int \rho \log \rho  d {\sigma} , &  {\rm if} \  m= \rho \cdot {\sigma} \\
             + \infty, &  \ {\rm otherwise} .
           \end{array}\right.
 $$ 
\begin{theorem}
The minimizers of the functional
\begin{equation} 
\label{sph-logBM}
 F(\nu) = \frac{1}{n} Ent(\nu) - K(\mu,\nu),
\end{equation}
are solutions to the log-Minkowski problem for $\mu$.
\end{theorem}

Note that (\ref{logh}) and (\ref{sph-logBM}) are in remarkable corespondence with (\ref{prekopa0}) and (\ref{GaussMM}).
Some other connections between the moment-measure problem and the log-Minkowski problem has been mentioned in \cite{CK}.
However, we should stress that unlike Gaussian or Euclidean case (see \cite{S}, \cite{KolKos}) it is not clear whether  $F$ is 
displacement convex. Note that the strong displacement convexity of $F$ would imply uniqueness of the solution to the
log-Minkowski problem.

  On the other hand, we were able to prove the following analog of the Gaussian transportation inequality (\ref{Tal}).
  \begin{theorem}
  Every symmetric measure $\nu$ on $S^{n-1}$  satisfies
  $$
   \frac{1}{n} Ent(\nu) \ge K({\sigma}, \nu).
   $$
   \end{theorem}
   This result seems to be a natural generalization of  (\ref{Tal}) for the sphere.
   Other transportational and functional  inequalities on the sphere has been studied in \cite{BChG}, \cite{CE} (see also \cite{BGL}).
   We emphasize that the standard proofs of  transportation inequalities usually involve displacement convexity 
  arguments (or some equivalent constructions).
   
  The proof of this inequality follows the classical arguments with an additional ingredient: with the help of the Blaschke-Santal\'o
  inequality   we establish  the following estimate which compensates the lack of displacement convexity
(for a  more general statement see Proposition \ref{trfh})
\begin{equation}
\label{trhs}
\frac{1}{n-1}\int_{S^{n-1}} {\rm Tr}(D^2 h)^{-1} dx \ge \int_{S^{n-1}} \frac{dx}{h},
\end{equation}
where $h \in C^2(S^{n-1})$ is a symmetric support function of a convex body. The tightness of this inequality immediately implies uniqueness
for the log-Minkowski problem for the case $\mu = \sigma$, which was shown first by W.~Firey \cite{Firey} (see explanations in the last section).   
Finally, we conjecture the following inequality generalizing (\ref{trhs}):
\begin{equation}
\frac{1}{n-1}\int_{S^{n-1}} {\rm Tr}(D^2 f)^{-1} (D^2 h) d \mu \ge \int_{S^{n-1}} \frac{h}{f} d \mu,
\end{equation}
where $\mu = \frac{1}{n |\Omega_h|} h \det D^2 h  \cdot \mathcal{H}^{n-1}$ and  $h, f \in C^2(S^{n-1})$ are symmetric  support functions of convex bodies.
 This inequality implies uniqueness of a solution to the general even log-Minkowski problem, provided the equality case  holds if and only if $f=ch$.
 
 The author is grateful to Emanuel Milman for fruitful discussions.

  \section{Preliminaries: variation of the Euclidean Kantorovich distance}

 In our work we deal with optimal transportation on the sphere and related variational problems.
  Before we consider the spherical case let us briefly explain the relevant  Euclidean  technique.

  Let 
$
\nabla \Phi
$ 
be the optimal transportation of  $\mu = \rho dx = e^{-V} dx$  onto $\nu = e^{-W} dx$.
By the change of variables formula
$$
V = W(\nabla \Phi) - \log \det D^2 \Phi.
$$
We will calculate the variation of $W_2$.
The formula we get is a particular case of a well-known result  (\cite{Villani}, Theorem 8.3). We include it for completeness of the picture.

 Given a function $v$ with zero $\mu$-mean  $\int v \rho \ dx=0$ consider the  variation of  $\mu$:
 $$
 \rho_{\varepsilon} = ( 1 + \varepsilon v) \rho.
  $$ 
  Let  $\nabla \Phi_{\varepsilon}$ be the optimal transportation of  $\rho_{\varepsilon} dx$  onto  $\nu = e^{-W} dx$.
  One has 
  $$
  \Phi_{\varepsilon} = \Phi + \varepsilon u + o(\varepsilon),
  $$
  where 
  $$
V - \log ( 1 + \varepsilon v) = W(\nabla \Phi_{\varepsilon}) - \log \det (D^2 \Phi_{\varepsilon}).
  $$
  This relation immediately implies 
  \begin{equation}
  \label{L}
   Lu :=  {\rm{Tr}} [(D^2\Phi)^{-1} D^2 u]  -  \langle \nabla W(\nabla \Phi), \nabla u \rangle =v. 
  \end{equation}
  
  One can check by direct computations that $L$ is a generator of the Dirichlet form 
  on the metric-measure space $(D^2 \Phi, \mu)$ (see \cite{Kol}):
\begin{lemma}
  $$
  \int v f \rho dx = - \int \langle (D^2 \Phi)^{-1} \nabla u, \nabla f \rangle \rho dx.
  $$
\end{lemma}
\begin{proof}
By the change of variables formula for any smooth function $g$
$$
\int g(\nabla \Phi_{\varepsilon})  ( 1 + \varepsilon v) \rho dx = \int g d \nu.
$$
Expanding in $\varepsilon$ at zero one gets
$$
\int g(\nabla \Phi )  v \rho dx + \int \langle \nabla g(\nabla \Phi), \nabla u \rangle  \rho dx =0. 
$$
Setting $f = g(\nabla \Phi )$ one gets the claim.
\end{proof}

\begin{proposition}
\label{020718}
Let $F(\rho_{\varepsilon}) = \int \langle x, \nabla \Phi_{\varepsilon} \rangle d \rho_{\varepsilon}$. One has
$$
    \frac{d}{ d \varepsilon} F(\rho_{\varepsilon})|_{\varepsilon=0} 
    =  \int \Phi v \rho  dx.
    $$
\end{proposition}
\begin{proof}
$$
    \frac{d}{ d \varepsilon} F(\rho_{\varepsilon})|_{\varepsilon=0} =  \frac{d}{ d \varepsilon}  \int \langle x, \nabla \Phi_{\varepsilon}  \rangle \rho_{\varepsilon}  dx|_{\varepsilon=0}
    = \int \langle x, \nabla u \rangle \rho dx + \int \langle x , \nabla \Phi \rangle v \rho dx.
 $$
By the previous Lemma
 \begin{align*}
  \int \langle x , \nabla \Phi \rangle v \rho dx & =
 - \int \langle (D^2 \Phi)^{-1} \nabla u, \nabla \bigl(  \langle x , \nabla \Phi \rangle \bigr) \rangle \rho dx
 = -\int \langle x, \nabla u \rangle \rho dx  \\& - \int \langle (D^2 \Phi)^{-1} \nabla u,  \nabla \Phi  \rangle \rho dx
  = -\int \langle x, \nabla u \rangle \rho dx  + \int \Phi v \rho  dx.
  \end{align*}
  The proof is complete.
  \end{proof}
  Proposition \ref{020718} can be used to compute the variation of  $$W_2^2(\mu,\nu) = \int (x - \nabla \Phi)^2 d \mu.$$
  \begin{corollary}\label{kdinf}
  The variation of the Kantorovich distance $W_2$ can be computed as follows:
  $$
    \frac{d}{ d \varepsilon} W_2^2(\rho_{\varepsilon},\nu)|_{\varepsilon=0}
    = \int (x^2 - 2 \Phi) v \rho dx = -2 \int \langle x - \nabla \Phi, (D^2 \Phi)^{-1} \nabla u \rangle d \mu.
  $$
  \end{corollary}
  The above computation is a particular case of the general expression for derivative of
  $W_2$ on the space $\mathcal{P}_2$ (see \cite{Villani}, Theorem 8.3). According to this result
  $$
  \frac{d}{ d t} W_2^2(\rho_{t},\nu)|_{t=0}
  = 2 \int \langle x - \nabla \Phi, \xi_0 \rangle d \mu,
  $$ 
  where $\rho_t$ is a family of probability densities satisfying
  \begin{equation}
  \label{rhotxit}
  \frac{\partial \rho_t}{\partial t} + {\rm div} \bigl( \rho_t \cdot \xi_t \bigr) =0
  \end{equation}
for some given velocity field $\xi_t$. The reader can check that Corollary
\ref{kdinf} follows from these formulae  for a field $\xi_t$ with initial velocity $\xi_0 =  - (D^2 \Phi)^{-1} \nabla u$
and (\ref{rhotxit}) is equivalent to another representation of $v$:
$$
v = e^{V} {\rm div} \bigl( (D^2 \Phi)^{-1} \nabla u \cdot e^{-V}\bigr).
$$


  \section{Spherical logarithimic Kantorovich functional}

 {\bf Notations.} In what follows   $|\Omega|$ is the volume of a convex  body $\Omega \subset \mathbb{R}^n$,  $|\partial \Omega|$ is the $(n-1)$-dimensional 
  Hausdorff measure of the boundary  $\partial \Omega$ of $\Omega$,
  $B$ is the  unit ball in $\mathbb{R}^n$ with center at the origin and $ S^{n-1}$ is the boundary of $B$.
  The $(n-1)$-dimensional Hausdorff measure  is denoted by 
  $\mathcal{H}^{n-1}$, the normalized probability uniform measure on $S^{n-1}$
  is denoted by
  $\sigma$. Note that $${\sigma}  = \frac{\mathcal{H}^{n-1}|_{S^{n-1}}}{n|B|}.$$ 
 Finally,  the integral of a function $f$ over $\partial \Omega$ with respect to the $(n-1)$-dimensional 
  Hausdorff measure on the boundary  $\partial \Omega$
  will be denoted either by 
  $$\int_{\partial \Omega} f dx$$ 
  or by
  $$\int_{\partial \Omega} f d\mathcal{H}^{n-1}.$$
  Given a support function $h$,  the corresponding convex body will be denoted by $\Omega_h$.

  Everywhere in this section $\Omega$ is a compact convex body containing the origin.
  For any given couple of probability measures $\mu, \nu$ on $S^{n-1}$ we define the  Kantorovich functional 
   $$
 K(\mu,\nu) = \min_{\pi \in \Pi(\mu,\nu)} \int_{(S^{n-1})^2} c(x,y) d \pi,
 $$ 
where we work with the following cost function:
 $$
 c(x,y) = \left\{\begin{array}{cc}
             \log \frac{1}{\langle x, y \rangle}, & \langle x, y \rangle > 0\\
             + \infty, &  \langle x, y \rangle \le 0.
           \end{array}\right.
 $$ 
 This function naturally appears in convex geometry. This  can be explained 
 by the fact that the radial 
 $$
 r(y) = \sup \{t : t y \in \Omega, t >0\}, \  y \in S^{n-1}
 $$ 
 and the support function 
 $$
 h(x) = \sup_{z \in \Omega} \langle x, z \rangle, \ x \in S^{n-1}
 $$ of a convex body $\Omega$ are related by a  Legendre-type transform
  $$
  h(x) = \sup_{y \in S^{n-1}} r(y) \langle x, y \rangle,   
  $$
  $$
  \frac{1}{r(y)} = \sup_{x \in S^{n-1}} \frac{\langle x, y \rangle}{h(x)}
  $$
  (see \cite{Schneider}).
In particularly, $\log h$ and $\log r$ satisfy a Kantorovich-type duality relation
$$
\log r(y) - \log h(x) \le \log \frac{1}{\langle x, y \rangle}.
$$

Applying this observation,  V.~Oliker \cite{Oliker} has shown that the solution to this transportation problem
 solves a classical problem in convex geometry known as Aleksandrov problem.
 Other applications of the cost function  $c$ see in \cite{OlikerGangbo}, \cite{KKN}.

 {\bf Aleksandrov problem:}    Given a probability measure $\nu$ on $S^{n-1}$ 
    find a convex body $\Omega$ with $0 \in \Omega$ such that $\nu$ is the image of the uniform measure $\sigma$ of $S^{n-1}$
    under the mapping $$T(y) = n_{\partial \Omega} \circ r(y) y,$$ where 
    $$n_{\partial \Omega} \colon \partial \Omega \to S^{n-1}$$
    in the Gauss map.
    
    Aleksandrov gave in \cite{Aleks} the following  sufficient condition  for existence of a solution to this problem.  
 
 \begin{theorem} [\cite{Aleks}] Aleksandrov problem admits a unique solution provided $\nu$
 satisfies the following assumption: for every spherically convex set $A \subset S^{n-1}, A \ne S^{n-1}$
 \begin{equation}
 \label{Al-in}
 \nu(A) < \sigma(A_{\pi/2}),
 \end{equation}
 where $A_{\pi/2} = \{y: dist(A,y) < \frac{\pi}{2}\}$ and $dist$ is the standard distance on $S^{n-1}$.
 \end{theorem}
 
 A transportational solution to this problem was obtained by V.~Oliker in \cite{Oliker}. He also proved well-posedness
 of the dual problem and the absence of the  duality gap.
His result was generalized later by J.~Bertrand in \cite{Bert}.
 In particularly, Bertrand constructed a transportational solution for a couple of probability measures
 $\mu =  f \cdot \sigma, \nu$ under the generalized Alexandrov-type assumption:
 \begin{equation}
 \label{Al-in2}
 \nu(A) < \mu(A_{\pi/2})
 \end{equation}
 (see Remark 4.9 in \cite{Bert}).
 
 In what follows 
 $$
 \nabla_{S^{n-1}}, \ \nabla^2_{S^{n-1}}
 $$
 denote the spherical gradient and the spherical Hessian accordingly.
 
 We will also use the following operator acting on the tangent space $TS^{n-1}(x)$
 at the point $x$
 $$
 D^2 f(x) = f(x) \cdot {\rm  Id} + \nabla^2_{S^{n-1}} f(x),
 $$
 where $Id$ is the identical mapping on $TS^{n-1}(x)$.
 In particular
  $$
 D^2 h( n_{\partial \Omega} )  = II^{-1}_{\partial \Omega},
 $$
where  $II_{\partial \Omega}$ is the second fundamental form of $\partial \Omega$ and $h$ is the support functional of $\Omega$
(see \cite{Schneider}).
 
    The optimal transportation mapping corresponding to $c$  and pushing forward $\mu$ onto $\nu$ has the form
   \begin{equation}
   \label{TH}
  T(x) = \frac{h(x) \cdot x + \nabla_{S^{n-1}} h(x)}{\sqrt{h^2(x) + |\nabla_{S^{n-1}} h(x)|^2}}.
  \end{equation}
  
  Assuming smoothness of $h$ one can verify the following change of variables formula.
  \begin{lemma}
   Assume that $T$ pushes forward a probability measure $\rho_1 \cdot \sigma$ onto another probability measure $\rho_2 \cdot \sigma$. 
  Then  the following change of variables formula holds
  \begin{equation}
  \label{MA}
  \rho_1 = \rho_2(T) \frac{h \cdot \det D^2 h }{(h^2 + |\nabla_{S^{n-1}} h|^2)^{\frac{n}{2}}}.
  \end{equation}
  \end{lemma}
  
  It is important to have in mind the following relation between $r$ and $h$:
  $$
  r(T) = \sqrt{h^2 +  |\nabla_{S^{n-1}} h|^2}
  $$
    (see \cite{Schneider}). In particularly, (\ref{MA}) reads also as
   \begin{equation}
  \label{MA1}
  \rho_1 =  \frac{\rho_2(T)}{r^n(T)} {h \cdot \det D^2 h }.
  \end{equation}

  The inverse mapping $S = T^{-1}$ has the form (see \cite{Oliker})
  $$
  S(y) = \frac{-\nabla_{S^{n-1}} r(y) + r(y) y}{\sqrt{|\nabla_{S^{n-1}} r(y)|^2 + r^2(y)}}.
  $$
  To see that $S$ satisfy the same equation (\ref{TH}) for an appropriate choice
  of the potential, one needs to pass to  the inverse radial function
  \begin{equation}
  \label{fr}
  f(y) = \frac{1}{r(y)}.
  \end{equation}
  One has
  $$
  S(y) = \frac{f(y) y + \nabla_{S^{n-1}} f(y)}{\sqrt{|\nabla_{S^{n-1}} f(y)|^2 + f^2(y)}}.
  $$
  
  In our work we will mainly  concentrate on the symmetric case, meaning that
  $\mu$ and $\nu$ are invariant with respect to $x \to -x$. This assumption implies that the body $\Omega$
  is symmetric and $h$ is even.
  In the symmetric case the Aleksandrov sufficient condition  is automatically satisfied except of some degenerate situations.
  
  \begin{lemma}
   \label{transport-symmetric}
  Assume that $\mu$ and $\nu$ are symmetric, $\mu = f \cdot \sigma$, $f$ is positive $\sigma$-a.e., and $\nu(S^{n-1} \cap L) = 0$ for every hyperplane $L$
  passing through the origin. Then   (\ref{Al-in2}) is satisfied and there exists a unique solution
  to the Aleksandrov problem.
  \end{lemma}
  \begin{proof}
  Let us check (\ref{Al-in2}).
  For every spherically convex subset $A \ne S^{n-1}$ one can find a hemisphere $S_{l} = S^{n-1} \cap \{l \ge 0\}$, where $l$
  is a linear functional such that $A \subset S^{l}$. The case  $\nu(A)=0$ is obvious, so one can assume $\nu(A)>0$. Since $\nu(S^{n-1} \cap \{l = 0\})=0$, one has $\nu(A) = \nu(A')$, where $A' = A \setminus \{l=0\}$.
  By the symmetry assumption one gets that $\nu(A') = \nu(-A')$. Since $A' \cap -A' = \emptyset$, one gets $\nu(A) = \nu(A') \le 1/2$. Next we note that  $\nu(A)>0$, hence one can find a two-points set $M = \{a,b\} \subset A$
  but the set $M_{\pi/2}$ is a union of two distinct hemispheres. Since $\mu$ is symmetric and admits positive density, one immediately gets
  $\mu(A_{\pi/2}) \ge \mu(M_{\pi/2})> 1/2 \ge \nu(A)$. 
  \end{proof}
  
  \begin{remark}
  \label{1606}
  Unlike the Aleksandrov problem, a solution to the Monge-Kantorovich problem for $c$ always exists, because the cost function $c$ is lower semicontinuous.
  It may happen that the dual solution $(h, r)$ does not define any compact convex body $\Omega$ and the total cost function
  equals $+\infty$. For instance, consider $n=3$, $\mu$ is the symmetric measure which gives value $1/2$ to every pole and $\nu$ is concentrated on 
  the equator. Clearly, in this case any transportation plan $\Pi$ is supported on the set $\{ c= +\infty\}$. 
  
  If the Aleksandrov problem admits a solution and  $\mu = f \cdot \sigma$, then  $c \in L^{\infty}(\Pi)$
  for the corresponding optimal transport plan $\Pi$ (see \cite{Bert}).
  \end{remark}
   
   \section{Variation of the spherical  log-functional}
   
   Let $$\mu = \rho_{\mu} \cdot \sigma= e^{-V} \cdot \sigma$$ and
  $$\nu =  \rho_{\nu} \cdot \sigma = e^{-W} \cdot \sigma$$  be probability measures on $S^{n-1}$ and $T$ be the optimal transportation mapping for the cost function $c(x,y)$.
  It will be assumed that  $\rho_{\mu}, \rho_{\nu}$ and $T$ are sufficiently smooth.  Moreover, we assume that $h$ satisfies 
  $$
  D^2 h(x) >0, \  \ \forall x \in S^{n-1}.
  $$
  
  For the fixed target measure $\nu$ let us consider the variation of the source measure
  $$
  \mu_{\varepsilon} =  (1 + \varepsilon v) \mu, \ \int v d \mu =0. 
  $$
  It corresponds to the following variation of the Kantorovich potential :
  $$
  h_{\varepsilon} = h + \varepsilon u + o(\varepsilon).
  $$
 The variation 
  of the mapping $T$ looks as follows : 
  $$
  T_{\varepsilon} = T + \frac{\varepsilon}{\sqrt{h^2 + |\nabla_{S^{n-1}} h|^2}}  {\rm Pr}_{TS^{n-1}_{T(x)}} (ux + \nabla_{S^{n-1}} u) + o(\varepsilon).
  $$
  Here 
  $$
  {\rm Pr}_{TS^{n-1}_{y}}
  $$
  is the projection of $\mathbb{R}^{n}$ onto the tangent space of $S^{n-1}$ at the point $y$.
  
  By the change of variables formula (\ref{MA}):
  $$
  (1 + \varepsilon v) \rho_{\mu} = \rho_{\nu}(T_{\varepsilon})   \frac{h_{\varepsilon} \cdot \det D^2 h_{\varepsilon} }{(h_{\varepsilon}^2 + |\nabla_{S^{n-1}} h_{\varepsilon}|^2)^{\frac{n}{2}}}.
  $$ 
  Performing the Taylor expansion of the right-hand side and applying, in particular, the relation 
  $$\det(A+ \varepsilon B) = 
  \det A \cdot ( 1 +  \varepsilon {\rm Tr [A^{-1} B]} + o(\varepsilon)) 
  $$ one obtains
\begin{equation}
\label{vlu}
  v = {\rm Tr}  (D^2 h)^{-1} D^2 u - \Big\langle \nabla_{S^{n-1}} W(T) + n T,
   \frac{ux + \nabla_{S^{n-1}} u}{\sqrt{h^2 + |\nabla_{S^{n-1}} h|^2}} \Big\rangle + \frac{u}{h}.
   \end{equation}
  Let us denote the right-hand side of (\ref{vlu}) by $$L_{\mu,\nu} \Bigl( \frac{u}{h}\Bigr).$$
  This operator is the spherical analog of the Euclidean operator $L$ (\ref{L}).
  As in the Euclidean case $L_{\mu,\nu}$ is assciated with a Dirichlet form on $S^{n-1}$.  
  
  \begin{theorem}
  \label{ibp}
$L_{\mu,\nu}$  generates the following symmetric Dirichlet form:
   $$
 \int_{S^{n-1}}  h { \Big\langle  \nabla_{S^{n-1}} g,  
   ( D^2 h)^{-1} \nabla_{S^{n-1}} \Bigl( \frac{u}{h} \Bigr) \Big\rangle} d \mu  = -  \int_{S^{n-1}} g  L_{\mu,\nu} \Bigl( \frac{u}{h}\Bigr) d \mu,
 $$
     \end{theorem}
     
  \begin{proof}
  Expanding the left-hand side of the change of variables formula in $\varepsilon$
  $$
  \int_{S^{n-1}} f(T_\varepsilon)  \ d \mu_{\varepsilon}  = \int_{S^{n-1}} f  d \nu
  $$
 one obtains
 $$
   \int_{S^{n-1}} f(T)  v d \mu + \int_{S^{n-1}}  \frac{ \langle \nabla_{S^{n-1}} f \circ T, {\rm Pr}_{TS^{n-1}_{T(x)}} (ux + \nabla_{S^{n-1}} u)   \rangle}{\sqrt{h^2 + |\nabla_{S^{n-1}} h|^2}} d \mu = 0.
 $$
 Set: 
 $$
 g = f(T).
 $$
 One has $\nabla_{S^{n-1}} g = (DT)^* \nabla_{S^{n-1}} f  \circ T$, hence
 $$
 \nabla_{S^{n-1}} f \circ T = [(DT)^* ]^{-1} \nabla_{S^{n-1}} g. 
 $$
 Consequently
 $$
   \int_{S^{n-1}} g  v d\mu + \int_{S^{n-1}}  \frac{ \langle  \nabla_{S^{n-1}} g,  (DT)^{-1}  {\rm Pr}_{TS^{n-1}_{T(x)}} (ux + \nabla_{S^{n-1}} u)   \rangle}{\sqrt{h^2 + |\nabla_{S^{n-1}} h|^2}} d \mu = 0.
 $$
 It is easy to verify that
  $$
  DT|_{TS^{n-1}_{x}} = \frac{1}{\sqrt{h^2 + |\nabla_{S^{n-1}} h|^2}} {\rm Pr}_{TS^{n-1}_{T(x)}} D^2 h.
 $$
 This formula implies
 \begin{align*}
 & \frac{ \big\langle  \nabla_{S^{n-1}} g,  (DT)^{-1}  {\rm Pr}_{TS^{n-1}_{T(x)}} (ux + \nabla_{S^{n-1}} u)   \big\rangle}{\sqrt{h^2 + |\nabla_{S^{n-1}} h|^2}} 
\\&  =
  \big\langle  \nabla_{S^{n-1}} g,  \bigl[ {\rm Pr}_{TS^{n-1}_{T(x)}} D^2 h \bigr]^{-1}  {\rm Pr}_{TS^{n-1}_{T(x)}} (ux + \nabla_{S^{n-1}} u)   \big\rangle
 \end{align*}
 We need to compute
 $$
 v_1 = \bigl[ {\rm Pr}_{TS^{n-1}_{T(x)}}  D^2 h \bigr]^{-1}  {\rm Pr}_{TS^{n-1}_{T(x)}} (u x),
 $$
 $$
 v_2 = \bigl[ {\rm Pr}_{TS^{n-1}_{T(x)}}  D^2 h  \bigr]^{-1}  {\rm Pr}_{TS^{n-1}_{T(x)}} \bigl( \nabla_{S^{n-1}} u \bigr).
 $$
 Since $\nabla_{S^{n-1}} u \in TS^{n-1}_{x}$, it is easy to check that
 $$
 v_2 =  \bigl(  D^2 h   \bigr)^{-1}  \nabla_{S^{n-1}} u.
 $$
Let us compute $ v_1 =  u\bigl[ {\rm Pr}_{TS^{n-1}_{T(x)}}  D^2 h  \bigr]^{-1}  {\rm Pr}_{TS^{n-1}_{T(x)}} x $. To this end we note that
$v_1$ is the unique vector from $TS_x$ such that $\omega = \frac{1}{u} \cdot D^2 h  \cdot v_1$
belongs to $TS_x$ and satisfies
$$
\omega = x + \alpha T
$$
for some $\alpha \in \mathbb{R}$. The condition $\langle T(x), \omega \rangle =0$ implies
$$
\omega = x- \frac{T(x)}{\langle x, T(x) \rangle} = - \frac{\nabla_{S^{n-1}} h}{h}
$$
and
$$
v_1 = - u  (D^2 h)^{-1} \frac{\nabla_{S^{n-1}} h}{h}.
$$
Finally
$$
   \int_{S^{n-1}} g  v d \mu + \int_{S^{n-1}}  { \big\langle  \nabla_{S^{n-1}} g,  
    ( D^2 h )^{-1}  \Bigl( \nabla_{S^{n-1}} u  - u \frac{\nabla_{S^{n-1}} h}{h}  \Bigr) \big\rangle} d \mu = 0.
 $$
 Equivalently
 $$
   \int_{S^{n-1}} g  v d \mu  = - \int_{S^{n-1}}  h { \big\langle  \nabla_{S^{n-1}} g,  
    \bigl(  D^2 h  \bigr)^{-1} \nabla_{S^{n-1}} \Bigl( \frac{u}{h} \Bigr) \big\rangle} d \mu.
 $$
 The proof is complete.
 \end{proof}
 
 Thus we obtain that $L_{\mu,\nu}$
 is the generator of 
 $$
 \mathcal{E}_{\mu,\nu} (f,g) = \int_{S^{n-1}}  h { \Big\langle    
   ( D^2 h)^{-1} \nabla_{S^{n-1}} f,  \nabla_{S^{n-1}} g\Big\rangle} d \mu.
 $$
 The quadratic form
 $$
 g_h = \frac{D^2 h}{h} = {Id} + \frac{\nabla_{S^{n-1}}^2 h}{h}
 $$
 is a Riemannian metric on $S^{n-1}$ naturally associated with the couple $(\mu, \nu)$. For this metric $\mathcal{E}_{\mu,\nu}$
 is the standard weighted energy form:
 $$
 \mathcal{E}_{\mu,\nu} (f) = \int \|\nabla_{{g_h}} f\|^2_{g_h} d \mu
 $$
 and $L_{\mu,\nu}$ is the weighted Laplacian.
 
 \begin{lemma}
 The variation $$\delta_v 
 K(\mu,\nu)  =
 \lim_{\varepsilon \to 0} \frac{K(\mu_{\varepsilon},\nu) - K(\mu,\nu)}{\varepsilon},
 $$
 where 
 $
  \mu_{\varepsilon} =  (1 + \varepsilon v) \mu, \ \int_{S^{n-1}} v d \mu =0,
  $
 of the functional
 $
 K(\mu,\nu)
 $
 satisfies
 $$
 \delta_v  K(\mu,\nu)  = - \int_{S^{n-1}} \log h  \cdot v  d \mu.
 $$
 \end{lemma}
 \begin{proof}
 Note that
 $$
 K(\mu,\nu) = \int_{S^{n-1}} \log \frac{\sqrt{h^2 + |\nabla_{S^{n-1}} h|^2}}{h} d \mu.
 $$
 One gets
 \begin{align*}
 \delta_v 
 K(\mu,\nu)    & = - \int_{S^{n-1}} \frac {u}{h} d\mu  + \int_{S^{n-1}} \frac{u h + \langle \nabla_{S^{n-1}} u, \nabla_{S^{n-1}} h \rangle}{{h^2 + |\nabla_{S^{n-1}} h|^2}}   d \mu
 - \int_{S^{n-1}} \log {h} \cdot v d \mu
 \\&  + \int_{S^{n-1}} \log {\sqrt{h^2 + |\nabla_{S^{n-1}} h|^2}} v d \mu.
 \end{align*}
 Note that by (\ref{ibp}) 
 \begin{align*}
 \int_{S^{n-1}} & \log {\sqrt{h^2 + |\nabla_{S^{n-1}} h|^2}} v d \mu
 = 
 -  \int_{S^{n-1}}  h  \frac{ \Bigl\langle D^2 h \cdot \nabla_{S^{n-1}} h, ( D^2 h)^{-1} \nabla_{S^{n-1}} \bigl( \frac{u}{h} \bigr) \Big\rangle }{{h^2 + 
 |\nabla_{S^{n-1}} h|^2}} d \mu
\\& = -  \int_{S^{n-1}}  h  \frac{ \Bigl\langle \nabla_{S^{n-1}} h,  \nabla_{S^{n-1}} \bigl( \frac{u}{h} \bigr) \Big\rangle }{{h^2 + |\nabla_{S^{n-1}} h|^2}} d \mu =  -  \int_{S^{n-1}}    
\frac{ \langle \nabla_{S^{n-1}} h,  \nabla_{S^{n-1}}  u \rangle }{{h^2 + |\nabla_{S^{n-1}} h|^2}} d \mu 
\\& +  \int_{S^{n-1}}    \frac{ u |\nabla_{S^{n-1}} h|^2}{(h^2 + |\nabla_{S^{n-1}} h|^2) h} d \mu.
  \end{align*}
  Substituting this into the expression for 
  $\delta_v K(\mu,\nu)$ one gets the claim.
 \end{proof}
 
 This result applied to the mapping $S=T^{-1}$
 gives the following formula (see (\ref{fr})):
 \begin{corollary}
 \label{varnu}
 The variation $\delta_{\omega}
 K(\mu,\nu)  =
 \lim_{\varepsilon \to 0} \frac{K(\mu, \nu_{\varepsilon}) - K(\mu,\nu)}{\varepsilon},
 $
 where 
 $$
  \nu_{\varepsilon} =  (1 + \varepsilon \omega) \nu, \ \int_{S^{n-1}} \omega d \nu =0,
  $$
 satisfies
 $$
 \delta_{\omega}  K(\mu,\nu)  = -\int_{S^{n-1}} \log f  \cdot \omega \ d \nu =  \int_{S^{n-1}} \log r   \cdot \omega \ d \nu.
 $$
 \end{corollary}

 \section{Variational formulations of the log-Minkowski problem}

 Given a convex body $\Omega$ containing the origin let us denote by $m$ the measure on $\partial \Omega$ which is the image of the Lebesgue measure
 on $\Omega$ under the mapping $x \to \frac{x}{\|x\|_{\Omega}}$, where $\| \cdot \|_{\Omega}$ is the associated norm.
 This measure  can be expressed as follows: 
 $$
 m = \frac{1}{n} \langle x, n_{\partial_{\Omega}} \rangle \cdot\mathcal{H}^{n-1}|_{\partial \Omega},
 $$
where $\mathcal{H}^{n-1}$ is the  $(n-1)$-dimensional Hausdorff measure. 
 
 \begin{definition}
 The image $\mathcal{C}_{\Omega}$ of $m$ under the Gauss map is called {\bf the cone measure of $\Omega$ }. 
 \end{definition}
 
Note that
 $$
\mathcal{C}_{\Omega} = \frac{1}{n} h \det D^2 h \cdot \mathcal{H}^{n-1}|_{S^{n-1}},
 $$
 provided $h$ is sufficiently regular.
 
 {\bf Log-Minkowski problem.} Given a probability measure $\mu$ on $S^{n-1}$ find a convex set $\Omega \subset \mathbb{R}^n$
 with $|\Omega|=1$ such that $\mu$ is the cone measure of $\Omega$.
 
 The analytical formulation of the log-Minkowski problem
 looks as follows: given a probability measure with density $$\mu = \rho_{\mu} \cdot \sigma =  \frac{\rho_{\mu}}{n|B|} \cdot \mathcal{H}^{n-1}|_{S^{n-1}}  $$ on the unit sphere find a convex set $\Omega$ of volume $1$ which 
 support function $h$ satisfies
 $$
\rho_{\mu} =  h \det D^2 h |B|.
 $$  
Note that $\rho_{\mu}$ is indeed a probability density:
$$
\int_{S^{n-1}} \rho_{\mu}  d \sigma = \frac{1}{n|B|} \int_{S^{n-1}} \rho_{\mu} dx = \frac{1}{n} \int_{S^{n-1}} h \det D^2 h  dx =  \frac{1}{n} \int_{\partial {\Omega}} \langle x, n_{\partial_{\Omega}}\rangle  dx = |\Omega|=1. 
$$
 
 {\bf Assumption.} In what follows we assume that $\mu, \nu$  are symmetric measures and $h$ is an even function.

A variational approach to the log-Minkowski problem was suggested in \cite{BLYZ} (Lemma 4.1).
It was shown that  a solution $h$ to the following variational problem
 $$
  \int_{S^{n-1}} \log h \ d \mu  \mapsto \min,
 $$
 considered in the class of symmetric support functions of convex bodies with volume $1$, is a support function of a body $\Omega$ 
 solving the log-Minkowski problem for $\mu$.
  
  In our work we propose another variational functional for the symmetric  log-Minkowski problem 
  and defined with the help of mass transportation.
  Our approach partially motivated by 
  the results of \cite{S}.

  Another important relation of the log-Minkowski problem to the mass transportation problem  implicitly appeared 
  in \cite{KM}.
 Let $h$ be a support function of some convex set $\Omega$ of volume 1.
   The   second-order elliptic operator  $L_{\Omega}$ defined by the following integration by parts formula
$$
\frac{1}{n-1} \int_{S^{n-1}} \langle (D^2 h)^{-1} \nabla_{S^{n-1}} f, \nabla_{S^{n-1}} g \rangle h^2 \det D^2 h d x
= - \int_{S^{n-1}}  g \bigl(L_{\Omega} f \bigr) h  \det D^2 h d x
$$
is closely related to the log-Minkowski  problem. It has been shown in \cite{KM} that the (infinitesimal) log-Minkowski problem is a spectral problem
for $L_{\Omega}$. 
   
   It is easy to see that  $L_{\Omega}$ is a particular case of the operator
     $(n-1) L_{\mu,\nu}$. It corresponds to the
      following  couple of measures:
     $$\mu = \frac{1}{n} h \det D^2 h \cdot \mathcal{H}^{n-1}|_{S^{n-1}}, \ \ \  \nu = \frac{1}{n}  r^n \cdot \mathcal{H}^{n-1}|_{S^{n-1}},$$
     where 
     $h$ is the support and $r$ is the radial function of $\Omega$. Clearly, $\nu$
   is the push-forward image of $\mu$ under the  mapping $$T= (h x + \nabla h)/\sqrt{h^2 + |\nabla_{S^{n-1}} h |^2 }.$$
   The exact expression for the operator $L_{\mu,\nu}$ takes the form
   \begin{equation}
   \label{Lcone}
   L_{\mu,\nu} \Bigl( \frac{u}{h} \Bigr) =  {\rm Tr}  (D^2 h)^{-1} D^2 u - (n-1) \frac{u}{h}.
   \end{equation}
   
   Operator (\ref{Lcone})  has been studied already by D.~Hilbert in his work on the Brunn-Minkowski inequality. The original motivation
   of E.~Milman and the author was to study infinitesimal versions of the Brunn-Minkowski inequality, which are inequalities of the Poincar{\'e} type. See in this respect 
   \cite{Colesanti}, \cite{KM1}, \cite{KM2}.
   
   Recal that the entropy of the probability measure 
 $$
 \nu = \rho_{\nu} \cdot  \sigma
 $$
 is given by 
  $$
  Ent(\nu) = \int \rho_{\nu} \log  \rho_{\nu}   d {\sigma}
  $$
  and  $ Ent(\nu)  = +\infty$ if $\nu$ has no density.

 The variational formulae proved in the previous section immediately give a formal proof that the minimal points of 
 the following entropic/transportational functonal are precisely the solutions to the log-Minkowski problem.
 We will give later a rigorous justification of this fact.

  \begin{proposition}
  \label{stationary}
  Let  $\nu$ be a stationary point of the functional 
  $$
  F(\nu) = \frac{1}{n}  Ent(\nu) -  K(\mu,\nu).
  $$
    Then  $\mu = C h \det D^2 h \cdot \sigma$, where $h$ is the potential of the transportation mapping pushing forward
    $\mu$ onto $\nu$.
  \end{proposition} 
  
  \begin{proof}
 Set: $\nu_{\varepsilon} = (1 + \varepsilon \omega) \nu$, $\int \omega d \nu =0$. Then the variation of $K(\mu,\nu)$ (see Lemma \ref{varnu}) is equal to $ \int \omega \log r   d\nu $ and the variation of
 $Ent(\nu)$ is equal to $\int \omega \log\rho_{\nu} d \nu$.
 If $\nu$ is a stationary point, there holds:
 $$
 n \log r =   \log \rho_{\nu}  +  c.
 $$
 Thus $\nu = \rho_{\nu} \cdot \sigma =  e^{-c} r^n \cdot \sigma$. The expression for $\mu$ follows from the change of variables  formula
  and the normalization assumption: $\mu(S^{n-1})=1$.
 \end{proof}

Our observation has the following Euclidean companion.
The functional $F$ is a spherical analog of the functional  (\ref{KEfunc}),
while $
 h \mapsto \int_{S^{n-1}} \log h d \mu $ seems to be similar to (\ref{prekopa}).

Given a  probability measure $\nu = \varrho dx$ on $\mathbb{R}^n$ one can try to find
a log-concave measure $\mu = e^{-\Phi} dx$ (i.e., $\Phi$ is a convex function) 
satisfying the following remarkable property:
$\nu$~is the image of $\mu$ under the mapping $T$ 
generated by the logarithmic gradient of~$\mu$:
$$
T(x) = \nabla \Phi(x),
\quad 
\nu = \mu \circ T^{-1}.
$$
Following the terminology from \cite{CK},
 we say that  $\nu$ is a moment measure if such a function $\Phi$ exists.
 The associated Monge--Amp\`ere equation looks as follows :
$$
e^{-\Phi} = \varrho(\nabla \Phi) \det D^2 \Phi.
 $$ 
 It   is known that 
$\Phi$ is a maximum point of the following functional:
\begin{equation}
\label{prekopa}
J(f) = \log \int_{\mathbb{R}^n} e^{-f^*} dx  - \int_{\mathbb{R}^n} f d \nu,
\end{equation}
where $f^*$ is the Legendre transform of $f$.

An alternative  viewpoint was suggested in \cite{S}: it was shown  that $\rho = e^{-\Phi} $ gives  minimum to the functional
\begin{equation}
\label{KEfunc}
\mathcal{F}(\rho) = -\frac{1}{2} W^2_2(\nu, \rho dx) 
+ \frac{1}{2} \int_{\mathbb{R}^n} x^2 \rho \ dx + \int_{\mathbb{R}^n} \rho \log \rho dx.
 \end{equation}
 Other remarks on relations between the log-Minkowski problem and the moment measures can be found in \cite{CK}.
 Futher developments related to transportational and stability inequalities see in \cite{FGJ}, \cite{KolKos}, \cite{Fathi2}.
 
 \section{Minimizers of the variational functional}
  
 It is known that functional (\ref{KEfunc})
  admits certain displacement convexity properties (see \cite{S}). Unfortunately,
  we don't know whether 
   $K(\mu,\nu)$ is displacement convex. To see that the standard arguments fail, let us assume that $h$ is sufficiently regular and  apply the change of variables formula
   $$
     \rho_{\mu} = \rho_{\nu}(T) \frac{h \cdot \det D^2 h }{(h^2 + |\nabla_{S^{n-1}} h|^2)^{\frac{n}{2}}}.
     $$
     Take logarithm
     $$
     \log \rho_{\mu} = \log \rho_{\nu}(T) + n \log \Bigl( \frac{h}{\sqrt{ h^2 + |\nabla_{S^{n-1}}  h|^2 }}\Bigr)
     + \log \det \Bigl( {\rm Id} + \frac{\nabla^2_{S^{n-1}} h}{h}\Bigr)
     $$
     and integrate with respect to $\mu$:
     $$
     \frac{1}{n} Ent(\nu) - K(\mu,\nu) = \frac{1}{n} Ent(\mu)
     - \frac{1}{n} \int_{S^{n-1}} \log \det \Bigl( {\rm Id} + \frac{\nabla^2_{S^{n-1}} h}{h}\Bigr) d \mu.
     $$
    It is clear that the function
     $$
     - \log \det \Bigl( Id + \frac{\nabla^2_{S^{n-1}} h}{h}\Bigr)
     $$
  is not convex  with respect to the natural interpolation $t \to t h_1 + (1-t) h_2$.

  Another representation can be obtained from the duality principle. Let $(h,r)$ be the solution 
  to the dual Kantorovich problem. Applying Kantorovich duality one gets
  $$
  K(\mu,\nu) = -\int_{S^{n-1}} \log h d \mu + \int_{S^{n-1}} \log r d \nu. 
  $$
  Hence
  \begin{align*}
    \frac{1}{n} Ent(\nu) -  K(\mu,\nu) &  =   \frac{1}{n} Ent(\nu) + \int_{S^{n-1}} \log h d \mu - \int_{S^{n-1}} \log r d \nu
   \\&  = \int_{S^{n-1}} \log h d \mu + \frac{1}{n} \int_{S^{n-1}} \log \Bigl( \frac{\rho_{\nu}}{r^n}\Bigr)  \rho_{\nu} d \sigma.
  \end{align*}
  Set $ m =\frac{1}{C} r^n \cdot \sigma$, where  the normalization constant equals
  $$
  C = \int_{S^{n-1}} r^n d \sigma = \int_{S^{n-1}} h \det D^2 h d \sigma  = \frac{1}{n|B|}\int_{\partial \Omega_h} h(n_{\partial \Omega}) dx
  =  \frac{|\Omega_h|}{|B|}.
  $$
 
 Finally,
  \begin{align*}
      \frac{1}{n} Ent(\nu) - K(\mu,\nu) &
      = \int_{{S^{n-1}}} \log h d \mu - \frac{1}{n} \log C +  \frac{1}{n} \int_{S^{n-1}} \log \bigl( \frac{\rho_{\nu}}{r^n/C}\bigr) d \nu.
  \\& = \int_{S^{n-1}} \log h d \mu - \frac{1}{n}\log \int_{S^{n-1}} r^n d\sigma + \frac{1}{n} Ent_{m}(\nu),
  \end{align*}
where $ Ent_{m}(\nu) =\int_{S^{n-1}} \log \frac{d \mu}{ dm} d\mu \ge 0 $. 

\begin{theorem}
\label{010718}
Assume that $\mu(L \cap S^{n-1})=0$ for every hyperplane $L$ containing the origin. Then
the functional
$$
 F(\nu) = \frac{1}{n} Ent(\nu) - K(\mu,\nu)
$$ considered on the space of measures with $\rho_{\nu} \in L^{\infty}(\sigma)$ attains its minimum at some point.

For every such point the related transportational potential $h$ is a support function  of a symmetric convex body which after a suitable renormalization
gives minimum to 
$$
F_0(h) = \int_{S^{n-1}} \log h d \mu, \ |\Omega_h|=1.
$$
In particular, $h$ is a solution  to the log-Minkowski problem.
\end{theorem}
\begin{proof}
Set $m_{F} = \inf_{\nu} F(\nu)$. We want to estimate $m_{F}$ from below. Clearly, it is suffient to estimate
$F(\nu)$ from below for a positive density $\rho_{\nu}$ .   In this case the Aleksandrov problem admits a solution $(h, r)$ by Lemma \ref{transport-symmetric}  and, as we have already seen,
 $$\frac{1}{n} Ent(\nu) - K(\mu,\nu) = \int_{S^{n-1}} \log h d \mu - \frac{1}{n}\log \int_{S^{n-1}} r^n d\sigma + \frac{1}{n} Ent_{m}(\nu)
 .$$

Note that potential $h$ which defines optimal transportation $T$ 
is uniquely defined up to multiplication by a constant $\lambda>0$.
Thus without loss of generality one can normalize $h$ is such a way that
$|\Omega_h|=1$. Under this normalization
\begin{equation}
\label{030718-0}
      \frac{1}{n} Ent(\nu) - K(\mu,\nu)
      =  \int_{S^{n-1}} \log h d \mu + \frac{1}{n} \log |B| +  \frac{1}{n} Ent_m(\nu).
\end{equation}
Using that the term $Ent_m(\nu)$ is non-negative, one gets
$$
m_{F} \ge \min_{h: |\Omega_h|=1}  \int_{S^{n-1}} \log h d \mu + \frac{1}{n} \log |B|.
$$
Now let $h$ be a minimizer  of $F_0$. It follows from the result of \cite{BLYZ} that  $h$ exists and $h$ is a solution to the log-Minkowski problem, hence the push-forward measure $\mu \circ T^{-1}$ of $\mu$ under $T$ satisfies $\rho_{\nu} = c {r^n}$ and $Ent_m(\nu) =0$.
Thus for this $\nu$ one gets
$$
  \frac{1}{n} Ent(\nu) - K(\mu,\nu)
      =  \min_{h: |\Omega_h|=1}  \int_{S^{n-1}} \log h d \mu  + \frac{1}{n} \log |B|.
$$
We obtain that $\nu$ is a minimizer for $F$
and 
\begin{equation}
\label{030718}
m_F =  \min_{h: |\Omega_h|=1}  \int_{S^{n-1}} \log h d \mu  + \frac{1}{n} \log |B|.
\end{equation}

Now let $\tilde{\nu}$ be any minimizer of $F$. From the  relations  (\ref{030718-0}) and (\ref{030718}) we infer that $Ent_m(\tilde{\nu})=0$, where
$m=c r^n \cdot \sigma$ and $(h,r)$ is a solution to the dual Kantorovich problem for $(\mu, \tilde{\nu})$.
We obtain from (\ref{030718}) that $h$ is a minimizer of $F_0$. This completes the proof.
\end{proof}

\section{Uniqueness for log-Minkowski problem and transportation inequalities}

In what follows we are given a symmetric  probability measure $\mu$ on $S^{n-1}$.
We assume throughout the section that
$\mu(L \cap S^{n-1})=0$ for every hyperplane containing the origin. In particular, $\mu$ satisfies the subspace concentration property.
According to \cite{BLYZ} there exists a solution to the log-Minkowski problem for $\mu$. The uniqueness of a solution is an open problem.
In this paper we propose a weaker conjecture.

{\bf Conjecture: } {\it The functional
$$
 F(\nu) = \frac{1}{n} Ent(\nu) - K(\mu,\nu)
$$
has a unique global minimizer.}

Clearly, if $F$ has distinct global minimizers, then they all are solutions to the log-Minkowski problem according to Theorem \ref{010718}.

The representation 
\begin{equation}
\label{EK}
     \frac{1}{n} Ent(\nu) - K(\mu,\nu) = \frac{1}{n} Ent(\mu)
     - \frac{1}{n} \int_{S^{n-1}} \log \det \Bigl( Id + \frac{\nabla^2 h}{h}\Bigr) d \mu.
    \end{equation}
    obtained above
     can be used to compute  the second variation of $F$. Assuming that $h$ is sufficiently regular and using the second-order necessary condition for minimum
     one can easily get.
\begin{theorem} 
\label{inf-unique}
(Infinitesimal uniqueness).
Assume that the log-Minkowski problem admits a unique solution $\Omega_h$ for $\mu$, and $h$ is twice continuously differentiable.
Then every  even   function $u \in C^2(S^{n-1})$  satisfies the following inequality
$$
\int_{S^{n-1}} \| (D^2 h)^{-1} D^2 u\|^2_{HS} d \mu \ge (n-1) \int_{S^{n-1}} \Bigl( \frac{u}{h}\Bigr)^2 d \mu.
$$
\end{theorem}

\begin{proposition}
\label{trfh}
For any given couple of support functionals $h, f \in C^2(S^{n-1})$, $n>2$, the probability measure  
$$
\mu = \frac{h \det D^2 h}{n |\Omega_h|} \cdot \mathcal{H}^{n-1}|_{S^{n-1}}
$$
satisfies the following inequality
$$
\frac{1}{n-1} \int_{S^{n-1}} {\rm Tr}(D^2 f)^{-1} (D^2 h) d \mu   \ge \Bigl(\frac{|\Omega_h|}{|\Omega_f|} \Bigr)^{\frac{1}{n-1}} \Bigl( \int_{S^{n-1}}  \frac{h}{f} d \mu \Bigr)^{-\frac{1}{n-1}}.
$$
\end{proposition}
\begin{proof}
We start with the H\"older inequality
$$
\int_{S^{n-1}} {\rm Tr}(D^2 f)^{-1} (D^2 h) d \mu  \cdot
\int_{S^{n-1}} \frac{d \mu}{{\rm Tr}(D^2 f)^{-1} (D^2 h) }   \ge 1.
$$
Applying the arithmetic-geometric inequality and   H\"older inequality again one gets
\begin{align*}
&
\int_{S^{n-1}} \frac{1}{{\rm Tr}(D^2 f)^{-1} (D^2 h)  } d \mu  \le 
\int_{S^{n-1}} \frac{1}{ (n-1) \det^{\frac{1}{n-1}} [(D^2 f)^{-1} (D^2 h)]  } d \mu
 =  
\\& = \frac{1}{(n-1) n|\Omega_h|} \int_{S^{n-1}}  (\det D^2 f)^{\frac{1}{n-1}}  (\det D^2 h)^{\frac{n-2}{n-1}}  \ h \ d x
\\& \le  \frac{1}{(n-1) n|\Omega_h|} \Bigl(  \int_{S^{n-1}}f  \det D^2 f d x \Bigr)^{\frac{1}{n-1}}
\Big( \int_{S^{n-1}}  \frac{h^{ \frac{n-1}{n-2}}}{f^{\frac{1}{n-2}}} \det D^2 h \ d x\Bigr)^{\frac{n-2}{n-1}}.
\end{align*}
The integral  relation $ \int_{S^{n-1}} f  \det D^2 f d x = n |\Omega_f|$  implies
\begin{align*}
\int_{S^{n-1}} & \frac{d \mu}{{\rm Tr}(D^2 f)^{-1} (D^2 h)  }    \le 
 \frac{(  n |\Omega_f|)^{\frac{1}{n-1}}}{(n-1) n|\Omega_h|} 
\Big( \int_{S^{n-1}}  \frac{h^{\frac{1}{n-2}}}{f^{\frac{1}{n-2}}}   h \det D^2 h d x \Bigr)^{\frac{n-2}{n-1}}
\\& =  \frac{1}{n-1} \Bigl(\frac{|\Omega_f|}{|\Omega_h|} \Bigr)^{\frac{1}{n-1}}
\Big( \int_{S^{n-1}}  \frac{h^{\frac{1}{n-2}}}{f^{\frac{1}{n-2}}} d \mu \Bigr)^{\frac{n-2}{n-1}}
 \le  \frac{1}{n-1} \Bigl(\frac{|\Omega_f|}{|\Omega_h|} \Bigr)^{\frac{1}{n-1}}
\Bigl( \int_{S^{n-1}}  \frac{h}{f} d \mu \Bigr)^{\frac{1}{n-1}}.
\end{align*}
This  claim follows.
\end{proof}

Let us consider the case of
the uniform probability measure on $S^{n-1}$:
$$
\mu = \sigma. 
$$
The representation 
(\ref{EK})  takes the form
$$
     \frac{1}{n} Ent(\nu) - K({\sigma},\nu) =
     - \frac{1}{n} \int_{S^{n-1}} \log \det \Bigl( Id + \frac{\nabla^2 h}{h}\Bigr) d {\sigma}.
$$

\begin{theorem}
\label{210518}
The following inequality holds for every symmetric probability measure $\nu$:
\begin{equation}
\label{SphTal}
   \frac{1}{n} Ent(\nu) \ge K({\sigma},\nu).
\end{equation}

Equivalently, every even twice continuously differentiable support function $h$ satisfies
\begin{equation}
\label{leb-log}
 \int_{S^{n-1}} \log \det \Bigl( Id + \frac{\nabla^2 h}{h}\Bigr) d \sigma \le 0.
\end{equation}

Moreover, a stronger inequality holds:
\begin{equation}
\label{050518}
\frac{1}{n-1}\int_{S^{n-1}} {\rm Tr}(D^2 h)^{-1} d\sigma \ge \int_{S^{n-1}} \frac{d\sigma}{h}.
\end{equation}

Equalities  in (\ref{210518}),  (\ref{leb-log}), and  (\ref{050518}) hold only for constant $h$. 
\end{theorem}

\begin{remark}
Inequality (\ref{SphTal}) is a spherical variant of the Gaussian Talagrand transportation inequality 
$$
Ent_{\gamma}(\nu) \ge \frac{1}{2} W^2_2(\gamma, \nu),
$$
where $\gamma$ is the standard Gaussian measure.

In order to compare both inequalities let us note that
$$
\log \frac{1}{\langle x, y \rangle} = \log \frac{2}{2 - |x-y|^2}.
$$
Using convexity of the function $\log \frac{2}{2-t}$ and the Jenssen's inequality one gets
$$
     \frac{1}{n} Ent(\nu) \ge K(\nu,\sigma) \ge \log \frac{2}{2 - {\tilde{W}^2_2(\sigma, \nu)}}.
$$
Equivalently
$$
\frac{1}{2} {\tilde{W}^2_2(\sigma, \nu)} \le 1 - e^{\frac{1}{n} Ent(\nu) }.
$$
Here $$\tilde{W}^2_2 (\sigma,\nu) = \inf_{\pi \in \Pi(\sigma,\nu)} \int_{S^{n-1} \times S^{n-1}} |x-y|^2 d \pi$$ 
(note that  this is not the standard  Kantorovich functional $W_2$ because the cost function is equal to the squared Euclidean distance).

Analogously, for the Riemannian $l^1$-transportation cost 
$$W_1 (\sigma,\nu) = \inf_{\pi \in \Pi(\sigma,\nu)} \int_{S^{n-1} \times S^{n-1}} d(x,y) d \pi,
$$ where $d(x,y)$ is the standard Riemannian distance on $S^{n-1}$, one gets
\begin{align*}
\frac{1}{n} Ent(\nu) & \ge - \int_{S^{n-1}} \log \cos(d(x,T(x)) d \nu \ge - \log \cos  \int_{S^{n-1}} d(x,T(x) d \nu 
\\& \ge - \log \cos  W_1 (\sigma,\nu).
\end{align*} Hence
$$
W_1 (\sigma,\nu)  \le \arccos e^{-\frac{1}{n} Ent(\nu)}.
$$
\end{remark}
 
 \begin{remark}
Another transportation inequality for the sphere has been obtained by D.~Cordero-Erausquin in \cite{CE} for general 
(non-symmetric) measures:
$$
W_{c}(f \cdot \sigma, \sigma) \le n - n \int f^{1 - \frac{1}{n}} d \sigma,
$$
where $$c(x,y) = n - \frac{\sin^{n-1}(d)}{S^{n-1}_n(d)} - (n-1) \frac{S_n(d)}{\tan(d)}$$
and $S_n(d) = \Bigl(n \int_0^d \sin^{n-1}(s) ds\Bigr)^{\frac{1}{n}}$.
\end{remark}


\begin{proof}
Clearly, (\ref{SphTal})  follows from   (\ref{leb-log}) by the standard approximation arguments. 
Let us  derive  (\ref{leb-log})  from  (\ref{050518}).
Set:
$$
f(t) = \int_{S^{n-1}} \Bigl[ \log \det ( (1 + t(h-1)) I + t \nabla^2 h) - (n-1) \log ( 1 + t(h-1)) \Bigr] d\sigma.
$$
Note that (\ref{leb-log}) is equivalent to $f(1) \ge 0$. Since $f(0)=0$,
to prove  (\ref{leb-log}) it is sufficient to show that $f'(t) \le 0$, equivalently
$$
\int_{S^{n-1}} {\rm Tr} ( (1 + t(h-1)) I + t \nabla^2 h)^{-1} [ (h-1) I + \nabla^2 h] d\sigma \le (n-1) \int_{S^{n-1}} \frac{h-1}{1+t(h-1)} d\sigma. 
$$
Set $h_t = 1-t + t h$. Then the above inequality can be rewritten as follows:
$$
\frac{1}{t} \int_{S^{n-1}} {\rm  Tr} (D^2 h_t)^{-1} [  D^2 h_t - I] d\sigma \le \frac{(n-1)}{t} \int_{S^{n-1}} \frac{h_t-1}{h_t} d\sigma,
$$
which is equivalent to (\ref{050518}).

Let us  prove  (\ref{050518}).
Assume that  $n>2$. By Proposition \ref{trfh}
$$
\frac{1}{n-1} \int_{S^{n-1}} {\rm Tr}(D^2 h)^{-1} d \sigma  \ge \Bigl(\frac{|B|}{|\Omega_h|} \Bigr)^{\frac{1}{n-1}} \Bigl( \int_{S^{n-1}} \frac{d \sigma}{h}  \Bigr)^{-\frac{1}{n-1}},
$$
It remains  to prove that 
$$
\frac{|B|}{|\Omega_h|} \ge \Bigl( \int_{S^{n-1}}  \frac{d \sigma}{h} \Bigr)^n.
$$
Applying the volume formula for polar body $\Omega^{\circ}_h$ one gets
$$
 \Bigl( \int_{S^{n-1}}  \frac{d \sigma}{h}  \Bigr)^n \le  \int_{S^{n-1}}  \frac{d \sigma}{h^n}  = \frac{|\Omega^{\circ}_h|}{|B|}.
$$
The result follows from the Blaschke-Santal\'o inequality.

In the case $n=2$  inequality  (\ref{050518})  (after change of variables) reads as follows:
$$
\int_{\partial \Omega} k^2 dx \ge  \int_{\partial \Omega}  \frac{k}{\langle x, n_{\partial \Omega}\rangle} dx,
$$
where $k$ is curvature of $\partial \Omega$.
According to a result of M.E.~Gage \cite{Gage}
$$
\int_{\partial \Omega}  k^2 dx \ge \frac{\pi |\partial \Omega|}{\Omega}.
$$
By the Blaschke-Santal\'o inequality
\begin{equation}
\label{2dimBS}
  \int_{\partial \Omega}  \frac{k}{\langle x, n_{\partial \Omega} \rangle} dx = \int_{S^{n-1}} \frac{dx}{h} \le |S^1|^{1/2}   \sqrt{\int_{S^{n-1}} \frac{dx}{h^2}} 
  = |S^1|^{1/2}   \sqrt{ 2 |\Omega^{\circ}|}  \le 
 \frac{ 2 |\pi|^{\frac{3}{2}} }{\sqrt{|\Omega|}}.
\end{equation}
The result follows from the isoperimetric inequality.
\end{proof}
 
 It was shown by W.~Firey \cite{Firey} that  the log-Minkowski problem has the unique solution
for the case $\mu = \sigma$.  We give an independent proof of this result here.

\begin{theorem}
Assume that $h$ is a $C^2$-solution to the log-Minkowski problem for $\mu=\sigma$:
$$
h \det D^2 h =\frac{1}{|B|}, \ |\Omega_h|= 1.
$$
Then $\Omega_h$ is the ball of volume one.
\end{theorem}
\begin{proof}
Set $\nu = {r^n} |B| \cdot \sigma $, where $r$ is the radial function of $\Omega_h$.
According to  (\ref{Lcone}) and Theorem \ref{ibp}
$$
0  = \int_{S^{n-1}} L_{\mu,\nu} h d\mu = \int_{S^{n-1}} {\rm Tr}(D^2 h)^{-1} d \mu-  (n-1)\int_{S^{n-1}} \frac{d \mu}{h}.
$$
On the other hand, according to Theorem \ref{210518} equality in (\ref{050518}) holds only for constant $h$. This implies the claim.
\end{proof}

\begin{remark}
In the same way as for the uniform measure uniqueness of  solution  to the log-Minkowski problem for arbitrary regular  $\mu$ 
can be derived from the following conjectured inequality
\begin{equation}
\label{trfh-2}
\frac{1}{n-1}\int_{S^{n-1}} {\rm Tr}(D^2 f)^{-1} (D^2 h) d \mu \ge \int_{S^{n-1}} \frac{h}{f} d \mu,
\end{equation}
where $$\mu = \frac{1}{n |\Omega_h|} h \det D^2 h  \cdot \mathcal{H}^{n-1}|_{S^{n-1}},$$   $h, f$ are arbitrary even $C^2(S^{n-1})$ support functions,
provided the equality case holds if and only if $\frac{h}{f}$ is constant.
\end{remark}

Inequality (\ref{trfh-2}) is reminiscent of several inequalities appearing in relation to Minkowski-type problems.
First, note that the log-Brunn--Minkowski inequality is equivalent to the following inequality for support functions $h, f$ (see \cite{BLYZ0}):
\begin{equation}
\label{log-Min}
\int_{S^{n-1}} \log \Bigl( \frac{f}{h} \Bigr) d \mu \ge  \frac{1}{n} \log \frac{|\Omega_f|}{|\Omega_h|}.
\end{equation}
In addition, it was shown in  \cite{CLM} that the log-Brunn--Minkowski inequality has the following local version (see Proposition 4.4):
\begin{align*}
\int_{S^{n-1}} \frac{1 + {\rm Tr} (D^2 h)^{-1}  h}{h^2} u^2 d \mu - n \Bigl( \int_{S^{n-1}} \frac{u}{h} d \mu \Bigr)^2  \le 
\int_{S^{n-1}} \frac{\langle (D^2 h)^{-1} \nabla_{S^{n-1}} u, \nabla_{S^{n-1}} u \rangle}{h} d \mu,
\end{align*}
$u \in C^1(S^{n-1})$ (see \cite{KM} for a more general version for $p$-Brunn--Minkowski inequality).

\begin{remark}
One can conjecture a stronger inequality which by Proposition \ref{trfh} implies (\ref{trfh-2}) (and \ref{log-Min}):
\begin{equation}
\label{fh-false}
\frac{|\Omega_h|^{\frac{1}{n}}}{|\Omega_f|^{\frac{1}{n}}} \ge \int_{S^{n-1}} \frac{h}{f} d \mu.
\end{equation}
But this inequality does not hold even for $f=1$. Indeed, in this case 
(\ref{fh-false}) is equivalent to
$
\int_{S^{1}} h^2 \det D^2 h d x \le n \frac{|\Omega_h|^{1 + \frac{1}{n}}}{|B|^{\frac{1}{n}}}
$
and, finally, to
$
\int_{\partial \Omega_h} \langle x, n_{\partial \Omega} \rangle^2 d x \le \frac{2}{\sqrt{\pi}} {|\Omega_h|^{\frac{3}{2}}}
$
(for $n=2$). But this inequality fails for a thin long cylinder $[-1,1] \times [- R, R]$ for sufficiently large $R$.
  
  Note, however, that (\ref{trfh-2}) holds for $n=2$, $f=1$. Indeed, in this case (\ref{trfh-2}) reads as
  $$
  \int_{S^1} h (h + h'')^2 dx\ge \int_{S^1} h^2 (h + h'') dx . 
  $$
  Changing variables one gets that the latter is equivalent to
  $$
  \int_{\partial \Omega} \frac{\langle x, n_{\partial \Omega} \rangle}{k} dx \ge \int_{\partial \Omega} \langle x, n_{\partial \Omega} \rangle^2 dx,
  $$
  where $k$ is curvature of $\partial \Omega$. To prove it let us use a Bonnessen-type inequality  (see \cite{Gage}):
  $$
   \langle x, n_{\partial \Omega} \rangle |\partial \Omega| \ge |\Omega| + \pi  \langle x, n_{\partial \Omega} \rangle^2.
  $$
  Integrating over $\partial \Omega$  and applying $\int_{\partial \Omega}  \langle x, n_{\partial \Omega} \rangle  dx = 2 |\Omega|$ one gets
  $$
  \int_{\partial \Omega} \langle x, n_{\partial \Omega} \rangle^2 dx \le \frac{|\Omega| |\partial \Omega|}{\pi}.
  $$
  Using (\ref{2dimBS}) and Cauchy inequality one gets
  $$
   \frac{ 2 |\pi|^{\frac{3}{2}} }{\sqrt{|\Omega|}} \int_{\partial \Omega} \frac{\langle x, n_{\partial \Omega} \rangle}{k} dx  \ge   \int_{\partial \Omega} \frac{\langle x, n_{\partial \Omega} \rangle}{k} dx  \int_{\partial \Omega} \frac{k} {\langle x, n_{\partial \Omega} \rangle} dx \ge  |\partial \Omega|^2 .
  $$
  Finally
  $$
   \int_{\partial \Omega} \frac{\langle x, n_{\partial \Omega} \rangle}{k} dx
   \ge    \frac{|\partial \Omega|^2 \sqrt{|\Omega|}}{ 2 |\pi|^{\frac{3}{2}} }.
  $$
  The result follows from the  isoperimetric inequality.
  \end{remark}

\end{document}